\documentclass{amsart}
\usepackage{amssymb, latexsym}

\newtheorem{theorem}{Theorem}
\newtheorem{lemma}{Lemma}

\begin{document}
\title{On Geodesic Triangles in $\mathbb{H}^2$}
\author{Rita Gitik}
\address{ Department of Mathematics \\ University of Michigan \\ Ann Arbor, MI, 48109}
\email{ritagtk@umich.edu}
\date{\today}

\begin{abstract}
Let $M$ be an orientable hyperbolic surface without boundary and let $\gamma$ be a closed geodesic in $M$.
We prove that any side of any triangle formed by distinct lifts of $\gamma$ in $\mathbb{H}^2$ is shorter then $\gamma$.
\end{abstract}


\maketitle

\section{Introduction}

Behaviour of closed geodesics in hyperbolic surfaces has been a fruitful subject of research for many years. Such geodesics are often studied by looking at their lifts in covering spaces of the surface, cf. \cite{F-H-S}, \cite{Ga}, \cite{Gr}, \cite{H-S2}, \cite{H-S3}, \cite{NC},  and \cite{Gi}. In this paper we  consider three geodesics in $\mathbb{H}^2$ which are lifts of the same closed geodesic $\gamma$ in an orientable hyperbolic surface without boundary, cf. \cite{H-S1} and \cite{Gi}. We prove that if these three geodesics intersect to form a triangle then each side of that triangle is shorter than
$\gamma$. In contrast, a triangle in $\mathbb{H}^2$ formed by three arbitrary geodesic lines can have sides of any length.

The main result of this paper is the following theorem.

\begin{theorem}
Let $M$ be an orientable hyperbolic surface without boundary and let $\gamma$ be a closed geodesic in $M$. Any side of any triangle formed by distinct lifts of $\gamma$ in $\mathbb{H}^2$ is shorter then $\gamma$.
\end{theorem}

Theorem 1 is related to the results giving lower bounds on the angles of intersections of closed geodesics in hyperbolic surfaces, cf. \cite{Gil}. 
 
An important tool in studying geodesics in $\mathbb{H}^2$ is a tree $T$ in $\mathbb{H}^2$, defined in Section 2 of this paper,
cf.  \cite{H-S1}, pp. 111-112 and \cite{Gi}. It was shown in \cite{Gi} that no analogue of Theorem 1 holds in $T$.
 
The proof of Theorem 1 is given at in Section 5 of this paper. In order to prove Theorem 1, we need to prove the following special case first.

\begin{lemma}
Let $M$ be an orientable  non-compact hyperbolic surface without boundary which has finitely generated fundamental group and let $\gamma$ be a closed geodesic in $M$. Any side of any triangle formed by distinct geodesic lines in  the preimage of $\gamma$ in $\mathbb{H}^2$ is shorter then $\gamma$.
\end{lemma}

Note that a single-punctured hyperbolic sphere has a trivial fundamental group, so it does not have closed geodesics, hence Lemma 1 is vacuously true in this case.

A twice-punctured sphere  is homeomorphic to an annulus, so its fundamental group is infinite cyclic. Hence the preimage in $\mathbb{H}^2$ of any closed geodesic in a two-punctured sphere consists of a single geodesic line. It follows that in this case there are no triangles in $\mathbb{H}^2$ which satisfy the hypothesis of Lemma 1, so Lemma 1 is vacuously true in this case.

The proof of Lemma 1 for a hyperbolic single-punctured torus is given in \cite{Gi}.

The proof of Lemma 1 for all remaining cases is given in Section 4 of this paper.

\section{The Tree T in $\mathbb{H}^2$}

Excellent  expositions of hyperbolic geometry can be found in \cite{Be}, \cite{Le}, \cite{Ra}, \cite{Scott}, and \cite{St}.

Let $M$ be an orientable non-compact hyperbolic surface without boundary which has finitely generated fundamental group. Let the genus of $M$ be $k$ and let $l$ be the number of punctures in $M$. Let $n=2k+l-1$. There exist infinite simple disjoint geodesics $x_1, \cdots x_n$ in $M$ such that $M$ cut along the union of $x_i,1 \le i \le n$,  is an open two-dimensional disk $D$.  Also there exist closed geodesics  $y_1, \cdots , y_n$ in $M$  such that $x_i \cap y_i=$point and $x_i \cap y_j =\varnothing$ for $i \neq j$, which generate the fundamental group of $M$. Note that the fundamental group of $M$  is a free group of rank $n$. The universal cover of $M$ is the hyperbolic plane $\mathbb{H}^2$, so $M$ is the quotient of $\mathbb{H}^2$ by the action of $\pi_1(M)$. 

Let $\widetilde{D}$ be a lift of the disc $D$ to $\mathbb{H}^2$. Note that $\widetilde{D}$ is a $2n$-gon in $\mathbb{H}^2$.

Recall that an end of a surface without boundary and finitely generated fundamental group is homeomorphic to a product $S^1 \times [0,\infty )$. Hyperbolic surfaces without boundary and finitely generated fundamental group have two kinds of ends: a cusp end, which has finite area, and a flare end, which has infinite area. If all the ends of $M$ are cusps then $\widetilde{D}$ is an ideal $2n$-gon in $\mathbb{H}^2$. The action of $\pi_1(M)$ on $\mathbb{H}^2$ creates a tessellation of $\mathbb{H}^2$ by the translates of the closure of $\widetilde{D}$. 

Let $T$ be the graph in $\mathbb{H}^2$ dual to that tessellation, i.e. the vertices of $T$ are located one in each translate of $\widetilde{D}$, and each edge of $T$ connects two vertices of $T$ in adjacent copies of $\widetilde{D}$, so each edge of $T$ intersects just one lift of one $x_i$ in one point. As 
$\mathbb{H}^2$ is simply connected, $T$ is a tree. The tree $T$ can be considered to be the Cayley graph of the group $\pi_1(M)$ which is a free group of rank $n$ generated by the set $y_1, \cdots , y_n$. Define the distance $d_T(v,u)$ between two vertices $v$ and $u$ of $T$ to be the number of edges in a shortest path in $T$ connecting $v$ and $u$.

Any element $f$ of $\pi_1(M)$ acts on $T$ leaving invariant a unique geodesic line, called the axis of $f$. The vertices of this geodesic line are characterized as those which minimize $d_T(v, f(v))$. That minimum is called the translation length of $f$, and is equal to the length of the word $W$ in 
$\pi_1(M)= \langle y_1, \cdots ,y_n \rangle$, obtained from $f$ by reduction and cyclic reduction. Denote the length of the word $W$ in $\pi_1(M)$ by 
$L(W)$. Note that each oriented edge of $T$ is labeled by one of the generators $\{ y_i, 1 \le i \le n \}$ or their inverses $\{ y_i^{-1}, 1 \le i \le n \}$, so each oriented path in $T$ is labeled by a word in $\{ y_i, 1 \le i \le n \}$ and $\{ y_i^{-1}, 1 \le i \le n \}$ .

The following result is a generalization of Lemma 1 in \cite{Gi}. We include the proof for the sake of completeness  of this paper.

\begin{lemma}
 Let $f$ be an element in $\pi_1(M)= \langle y_1, \cdots , y_n \rangle$ and let $W$ be its reduced and cyclically reduced conjugate.
Consider two axes in the tree $T$ stabilized by $f$ and its conjugate $f'\in \pi_1(M)$. If those axes intersect in an interval labeled with a word $W_0$ such that $L(W_0)=L(W)-1$ then they coincide.
\end{lemma}

\begin{proof}
WLOG $W_0$ is an initial subword of $W$, hence WLOG there exists a decomposition $W=W_0x$, where $x$ is either a generator or an inverse of a generator in $\pi_1(M)$. Let $W'$ be a reduced and cyclically reduced conjugate of $f'$ containing $W_0$. Then the abelianization of $W$  implies that either $W'=xW_0$ or $W'=W_0x=W$. In either case, the intersection of the axes of $f$ and of $f'$ contains an interval of length $L(W)$, obtained by adding a single edge with label $x$ to an end  of the interval with label $W_0$. Hence the axes of $f$ and $f'$ coincide.
\end{proof}

\section{Quasi-Isometry of T and $\mathbb{H}^2$}

In this paper we work with the Poincar\'e disk model of the hyperbolic plane $\mathbb{H}^2$ and the Poincar\'e metric, given by 
$$ ds^2= \frac{4(da^2+db^2)}{(1-a^2-b^2)^2}$$
Any pair of points $a$ and $b$ in the Poincar\'e disk are joined by a unique geodesic, which is a part of the circle or the straight line passing through
$a$ and $b$ and orthogonal to the boundary of the Poincar\'e disc. The distance between $a$ and $b$ is given by
$$d(a,b)= 2 tanh^{-1} \frac{|a-b|}{1-\bar{b}a}$$
The topology defined by that metric on the Poincar\'e disc is equivalent to the Euclidean topology, but the Poincar\'e disc equipped with the Poincar\'e metric is complete.

We consider the standard metric on the hyperbolic surface $M$, given by the covering map from $\mathbb{H}^2$ to $M$.

Recall that a (not necessarily continuous) map $f$ from a metric space $(X_1, d_1)$ to a metric space $(X_2, d_2)$ is a quasi-isometry if there exists a constant $A \geq 1$ and non-negative constants $B$ and $C$  such that the following two conditions are satisfied.

\begin{enumerate}
\item
For any pair of points $p$ and $q$ in $X_1$, $$A d_1(p,q)- B \leq d_2(f(p),f(q)) \leq A d_1(p,q)+ B$$
\item
For any point $v \in X_2$ there exists a point $u \in X_1$ such that
$$ d_2(v, f(u)) \leq C$$
\end{enumerate}

Two metric spaces are called quasi-isometric if there exists a quasi-isometry between them.
Note that two compact metric spaces are always quasi-isometric.

Let $M$ be an orientable non-compact hyperbolic manifold without boundary such that the fundamental group of $M$ is finitely generated and let $T$ be the tree in $\mathbb{H}^2$ defined in the previous section.

Recall that a  map  $s: \mathbb{H}^2 \rightarrow T$ is $\pi_1(M)$-equivariant if  $g(s(x))=s(g(x))$ for any $g \in \pi_1(M)$ and $x \in \mathbb{H}^2$.

Consider the closure of the $2n$-gon $\widetilde{D}$ in $\mathbb{H}^2$, defined in the previous section. Its intersection with the tree $T$, which we denote by $T_D$, is a union of $2n$ geodesic segments which have one end in common. That endpoint is a vertex of the tree $T$, denote it by $v_D$. Define a $\pi_1(M)$-equivariant continuous map $s: \mathbb{H}^2 \rightarrow T$ as follows. Consider the $\epsilon$-neighborhood of the boundary of $\widetilde{D}$ for some small positive function $\epsilon$. Denote by $N_D$ the intersection of that neighborhood with $\widetilde{D}$.  For $x$ in the closure of 
$\widetilde{D}- N_D$ define $s(x)=v_D$. For $x \in N_D$ define $s(x)$ in two steps. First, project $N_D$ onto the intersection of $T_D$ with $N_D$. Second, stretch the image of the first step to fill the entire edge of the tree $T_D$ containing that image. Extend the map $s$ to the whole hyperbolic plane 
$\mathbb{H}^2$ by the action of the group $\pi_1(M)$. By construction, the map $s: \mathbb{H}^2 \rightarrow T$ is $\pi_1(M)$-equivariant and continuous. 

Note that the restriction of $s$ to any compact subset of $\mathbb{H}^2$ is a quasi-isometry even though $s$ might fail to be a quasi-isometry on the whole hyperbolic plane.

\section{Proof of Lemma 1}

Assume to the contrary that there exists a triangle $\Delta$ in $\mathbb{H}^2$ formed by geodesic lines $l,m$, and $n$, which are distinct lifts of the geodesic $\gamma$, such that the length of the side of $\Delta$ lying in $l$ is longer than $\gamma$. Note that $l$ is stabilized by some element $f$ in 
$\pi_1(F)$ which acts as a hyperbolic isometry of $\mathbb{H}^2$.

Let $T$ be the tree in $\mathbb{H}^2$ defined in Section 2 and let $W$ be a reduced and cyclically reduced word conjugate to $f$ in $\pi_1(M)$. 
Note that the geodesic lines $l, m$, and $n$  are  transversal to the lifts of the geodesics $x_i, 1 \le i \le n$ in $\mathbb{H}^2$.

Let $s: \mathbb{H}^2 \rightarrow T$ be the map defined in the previous section. Consider a very large disk $B$ in $\mathbb{H}^2$ which contains $\Delta$. As was explained in the previous section, the map $s$ restricted to $B$ is a quasi-isometry.
It can be arranged that the restriction of $s$ to the intersection of each of the lines $l, m$, and $n$ with $B$ is monotone, so $s$ maps those intersections onto  geodesics in $T$.

Lemma 2 implies that the length of any side of $s(\Delta)$ is strictly less then $2L(W)$. 

As $s$ is a quasi-isometry, the length of any side of $\Delta$ should be less than $A$(length of $ \gamma) + B$ for some constant $A \geq 1$ and a non-negative constant $B$. The careful analysis of the geometry of $\Delta$ given below, shows that $A=1$ and $B=0$, proving Lemma 1.

Indeed, let $p$ be the intersection of $l$ and $n$, and let $q$ be the intersection of $l$ and $m$. The length of $\gamma$ is equal to the length of the segment from $p$ to $f(p)$. As $f$ is an isometry, the length of that segment is equal to the length of the segment from $f(p)$ to  $f^2(p)$. 
By assumption, the segment from $p$ to $q$  is longer than $\gamma$, so the segment from $q$ to $f^2(p)$ is shorter than the segment from $p$ to $q$.
As $f$ is an isometry, the geodesics $n, f(n)$, and $f^2(n)$ make the same angle with $l$. Then as the segment from $q$ to  $f^2(p)$ is shorter than the segment from $p$ to $q$, the angle between $n$ and $l$ is equal to the angle between $f^2(n)$ and $l$, and the opposite angles between $m$ and $l$ are equal, it follows that $m$ and $f^2(n)$ intersect. 

Consider the intersections of the lifts of the geodesics $x_i, 1 \le i \le n$  with lines $l,m$, and $n$.
Let $b$ lifts of $x_i$ intersect both $l$ and $n$ to the left of the point $p$ and let $a$ lifts of $x_i$ intersect both $l$ and $n$ to the right of the point $p$. Then there are $a+b$ lifts of $x_i$ crossing $l$ and $n$, hence the length of the intersection $s(l) \cap s(n)$ is
$a+b$. Lemma 2 implies that $a+b < L(W)-1$. By a similar argument, the number $c$ of the lifts of $x_i$ intersecting both $l$ and $m$ is also less than $L(W)-1$. As $f$ is an isometry, there are $b$ lifts of $x_i$ crossing $l$ and $f^2(n)$ to the left of $f^2(p)$. Then the total number of the lifts of
$x_i$ crossing $l$ between the points $p$ and $f^2(p)$ is at most $a+b+c$, which is strictly less than $2L(W)$. However by construction, the number of the lifts of $x_i$ crossing  $l$ between the points $p$ and $f^2(p)$ should be equal to $2L(W)$. This contradiction  completes the proof of Lemma 1.

\section{Proof of Theorem 1}

Let $M$ be any orientable hyperbolic surface without boundary (possibly with infinitely generated fundamental group) and let $\gamma$ be a closed geodesic in $M$. Let $l, m$, and $n$ be distinct lifts of $\gamma$ to the hyperbolic plane which form a triangle. Let $\alpha$ generate the stabilizer of $l$ in 
$\pi_1(M)$. Let $g$ and $h$ be elements of $\pi_1(M)$ such that $m=gl$ and $n=hl$. Let $X$ be the cover of $M$ corresponding to the subgroup of $\pi_1(M)$  generated by $\alpha, g,$ and $h$. As $\pi_1(X)$ has $3$ generators, it follows that $X$ is an orientable not-compact hyperbolic surface without boundary which has finitely generated fundamental group. As $\pi_1(X)$ contains $\alpha$, it follows that $\gamma$ lifts to a closed geodesic $\gamma_X$ in $X$, and the geodesics $l, m,$ and $n$ are lifts of $\gamma_X$ to the hyperbolic plane. So applying Lemma 1, we obtain that each side of the geodesic triangle formed by $l, m,$ and $n$ is shorter that $\gamma_X$ which by construction has the same length as $\gamma$.

\section{Acknowledgment}

The author would like to thank Peter Scott for helpful conversations.

\end{document}